\documentclass[a4paper]{amsart}
\usepackage{a4}
\usepackage{amsthm,enumerate,amssymb,hyperref,color}

\DeclareMathOperator{\tr}{tr}
\DeclareMathOperator{\id}{Id}
\DeclareMathOperator{\cen}{Cen}
\DeclareMathOperator{\kew}{Skew}
\DeclareMathOperator{\ym}{Sym}
\DeclareMathOperator{\GM}{GM}
\DeclareMathOperator{\Span}{span}
\DeclareMathOperator{\Char}{char}
\newcommand{\csim}{\stackrel{\mathrm{cyc}}{\thicksim}}

\begin{document}

\newcommand{\A}{\cal{A}}
\newcommand{\F}{\cal{F}}
\newcommand{\R}{\cal{R}}
\newcommand{\cH}{\cal{H}}
\newcommand{\T}{\cal{T}}
\newcommand{\Z}{\cal{Z}}
\newcommand{\J}{\cal{J}}
\newcommand{\cS}{\cal{S}}
\newcommand{\M}{\cal{M}}
\newcommand{\K}{\cal{K}}
\newcommand{\C}{\cal{C}}
\newcommand{\U}{\cal{U}}
\newcommand{\B}{\cal{B}}
\newcommand{\V}{\cal{V}}
\newcommand{\cL}{\cal{L}}
\newcommand{\W}{\cal{W}}
\newcommand{\FF}{\mathbb F}
\newcommand{\KK}{\mathbb F_0}
\newcommand{\Q}{\mathbb Q}
\newcommand{\N}{\mathbb N}
\newcommand{\wh}[1]{\widehat{#1}}
\newcommand{\rsl}{\mathfrak{sl}}
\newcommand{\wt}{\widetilde}
\newcommand{\ov}{\overline}

\newcommand{\Skew}{{\kew\FF\langle\bar X,\bar X^*\rangle}}
\newcommand{\Sym}{{\ym\FF\langle\bar X,\bar X^*\rangle}}

\newcommand{\ax}{\langle\bar X\rangle}

\newtheorem{theorem}{Theorem}[section]
\newtheorem{proposition}[theorem]{Proposition}
\newtheorem{lemma}[theorem]{Lemma}
\newtheorem{corollary}[theorem]{Corollary}

\theoremstyle{definition}
\newtheorem{remark}[theorem]{Remark}
\newtheorem{definition}[theorem]{Definition}
\newtheorem{example}[theorem]{Example}
\newtheorem{examples}[theorem]{Examples}
\newtheorem{question}[theorem]{Question}

\newcommand\cal{\mathcal}

\title[Noncommutative Polynomials and Tracial Nullstellens\"atze]{
Values of Noncommutative Polynomials,\\ Lie Skew-Ideals and Tracial 
 Nullstellens\"atze}

\author[Matej Bre\v sar and Igor Klep]{Matej Bre\v sar${}^1$ and Igor Klep${}^2$}

\address{Department of Mathematics, Faculty of Mathematics and Physics, University
of Ljubljana, Slovenia, and \\
 Department of Mathematics and Computer Science, Faculty of Natural Sciences and Mathematics, University of Maribor, Slovenia	}
\email{matej.bresar@fmf.uni-lj.si}
\address{Department of Mathematics, University of California at San Diego,
USA}
\email{iklep@math.ucsd.edu}
\thanks{${}^1$Supported by the Slovenian Research Agency (program No. P1-0288).}
\thanks{${}^2$Supported by the Slovenian Research Agency (project No. Z1-9570-0101-06).}
\date{08 October 2008}

\begin{abstract}
A subspace of an algebra with involution is called a 
{\it Lie skew-ideal} if it is closed under Lie products with
\emph{skew-symmetric} elements. Lie skew-ideals are classified
in central simple algebras with involution (there are eight of them
for involutions of the first kind and four
for involutions of the second kind) and this classification result
is used to characterize noncommutative polynomials via their values in these
algebras. As an application, we deduce that a polynomial is
a sum of commutators and a polynomial identity of $d\times d$
matrices if and only if 
all of its values in the algebra of $d\times d$ matrices have zero trace.
\end{abstract}

\maketitle
\section{Introduction}

Interest in positivity questions of noncommutative polynomials
has been recently revived by Helton's seminal paper \cite{Hel}, in which he
proved that a polynomial is a sum of  squares of polynomials  if and
only if its values in matrices of any size
are positive semidefinite. A nice survey
of recent functional analytic results in this direction
and their various applications is given in \cite{dOHMP}.

One of the results in this vein was obtained  by the second author and
Schweighofer. They showed that
Connes' embedding conjecture on type II$_1$ von Neumann algebras
is equivalent to a problem of describing polynomials whose values at
tuples of self-adjoint $d\times d$ matrices (of norm at most $1$) have
nonnegative trace for every $d\ge 1$; see \cite[Theorem 1.6]{KS} for a
precise formulation.
The natural first step in understanding this problem is examining the zero
trace situation.  The authors proved that a polynomial
whose values always  have trace zero is a sum of commutators  \cite[Theorem
2.1]{KS}. This result was the initial motivation for the present work.

A non-dimensionfree approach to Connes' embedding conjecture entails
studying values
of polynomials when evaluated at tuples of $d\times d$
matrices for a \emph{fixed} $d$. A result in this spirit - a weak
version of  Helton's
sum of squares theorem -
can be obtained from the Procesi and Schacher 1976 paper \cite{ps}. A more
recent reference is
\cite{KU}, where
a Positivstellensatz characterizing polynomials whose values
in $d\times d$ matrices are all positive semidefinite is given.
The next step is to study the nonnegativity  of the trace,  and as a special
case the zero trace. The work on this paper begun by addressing the
latter problem. The solution, which we call the ``tracial
Nullstellensatz'', is simple:
a polynomial has zero trace when evaluated at $d\times d$
matrices if and only if it is a sum of commutators and a polynomial
identity of $d\times d$ matrices (Corollaries
\ref{C1} and
\ref{cor:C3}).

The zero trace problem has led us to consider the following more general
topic: {\em What is the linear span of all the values of a polynomial on a given
algebra $\A$?} Studying this question has turned out to be quite fruitful. As
we shall see, its answer yields tracial Nullstellens\"atze, and on the
other hand, we believe, admittedly  somewhat speculatively, that it is a
natural question 
related to various other
areas. As a matter of fact,  its consideration is, as we shall see,
connected to certain Lie structure topics and also to the notion of polynomial
identities.

Our crucial observation is that the linear span of values of a polynomial 
is a Lie
ideal of the algebra $\A$ in question (Theorem \ref{T1}). This paves the way
for the precise description. For example, in the special case where
$\A=M_d(\FF)$ is the  algebra of all $d\times d$ matrices  over a field
$\FF$ with  $\Char (\FF)=0$, Theorem \ref{T2}
implies that polynomials $f$
can be categorized  into four classes
according to their values:
\begin{enumerate}[\rm (i)]
\item
$f$ is a polynomial identity; in this case the span of its values is $0$;
\item
$f$ is a central polynomial; in this case the span of its values are the
scalar matrices;
\item
$f$ is a sum of commutators and a polynomial identity (but is not a
polynomial
identity); in this case the span of its values is the set of all trace zero
matrices;
\item
if $f$ is none of the above, then the span of its values is the entire
algebra $M_d(\FF)$.
\end{enumerate}
Theorem \ref{T2} works at a greater level of generality - it
is proved for prime PI algebras.
This class of algebras includes finite dimensional central simple 
algebras, and for
them the theorem is as clear as in the $M_d(\FF)$ case just stated.

Our main results, however, deal with algebras with involution. These theorems
are of
the same flavor as those outlined in the preceding paragraph, but somewhat
more involved. We consider noncommutative polynomials in $X_i$ and $X_i^*$
(i.e., elements of a free $*$-algebra), and observe that the linear span of
values of such a polynomial need not be a Lie ideal, but it is  always
closed under Lie products with \emph{skew-symmetric}
elements (Theorem \ref{prop:T1}). We call subspaces having this property
{\it Lie skew-ideals} and classify them
for
prime PI algebras (Theorems \ref{TPI} and \ref{TPIsecond}). Again, this enables us to categorize
polynomials into
classes, eight of them for an involution of the first kind and
four of them for an involution of the second kind (Theorems
\ref{cor:T2star} and \ref{thm:T2invo}).

The tracial Nullstellens\"atze mentioned above 
are deduced easily
from these results, cf.~cases (i) and (iii) above.
We revisit and reinterpret them in the setting of algebras
of generic matrices in the last section: an element of such an algebra
is a sum of commutators if and only if its trace is zero (recall that the
algebra
of generic matrices is a subalgebra of a matrix algebra over a polynomial
algebra and thus naturally equipped with a trace).

\section{The Lie Structure of Polynomial Values}

Let us fix the  notation that will be used throughout the paper. 
By  $\FF$ we denote  a field, and all our algebras will be algebras over $\FF$. Let $\A$ be an (associative) algebra. By $\Z$ we denote its center. If $\A$ is a $\ast$-algebra, i.e., an algebra with involution $\ast$, then by $\cS$ (resp.~$\K$) we denote   the set of all symmetric (resp.~skew-symmetric) elements in $\A$:
$$
\cS=\{a\in\A\mid a^* = a\},\quad
\K=\{a\in\A\mid a^* = -a\}.
$$
The advantage of this notation is brevity, but the reader should be warned against possible confusion. Let us point out that  $\cS$ and $\K$ depend on the involution; we will have the opportunity to consider different involutions on the same algebra $\A$ (cf.~Lemmas \ref{lmatriket} and \ref{lmatrikes}), so 
$\cS$ and $\K$ might differ from case to case.

\subsection{The involution-free case}

By $\FF\langle \bar{X}\rangle$ we denote the free algebra generated by $\bar{X} = \{X_1,X_2,\ldots\}$, i.e., the algebra of all polynomials in noncommuting variables $X_i$. 
Let  $\A$ be an algebra over $\FF$, and let $f = f(X_1,\ldots,X_n)\in \FF\langle \bar{X}\rangle$.
If $\cL_1,\ldots,\cL_n$ are subsets of $\A$, then by $f(\cL_1,\ldots,\cL_n)$ we denote the set of all values $f(a_1,\ldots,a_n)$ with $a_i\in \cL_i$, $i=1,\ldots,n$. If all $\cL_i$ are equal to $\A$,  then we simplify the notation and write $f(\A)$ instead of $f(\A,\ldots,\A)$. If $\U$ is a subset of $\A$, then by $\Span \U$ we denote the linear span of $\U$. 
One of the goals of this paper  is to describe $\Span f(\A)$ for all  polynomials $f$ and certain algebras $\A$. 
Of course it can happen that  $\Span f(\A)=0$ even when $f\ne 0$; such a polynomial $f$ is called a 
{\em $($polynomial$)$
identity} of $\A$. Algebras satisfying (nontrivial) polynomial identities are called {\em PI algebras}. This class of algebras includes all finite dimensional  algebras.

We say that a polynomial $f = f(X_1,\ldots,X_n)\in \FF\langle \bar{X}\rangle$ is {\em homogeneous in $X_i$} if each monomial of $f$
 has the same degree  with respect to $X_i$; if this degree is $1$, then we say that $f$ is {\em linear in $X_i$}. Further, we say that $f$ is
{\em multihomogeneous}  if it is  homogeneous in every $X_i$, $i=1,\ldots,n$. Every polynomial is a sum of multihomogeneous polynomials. A polynomial is said to be {\em multilinear}   if it is linear in every
$X_i$, $i=1,\ldots,n$. Thus, a multilinear polynomial in $X_1,\ldots,X_n$ is a linear combination of monomials of the form $X_{\sigma(1)}\ldots X_{\sigma(n)}$ where $\sigma$ is a permutation of $\{1,\ldots,n\}$. From the identity
\begin{align*}
[X_{\sigma(1)}\ldots X_{\sigma(n)},X_{n+1}] &= [X_{\sigma(1)},X_{n+1}]X_{\sigma(2)}\ldots X_{\sigma(n)}\\
+ X_{\sigma(1)} [X_{\sigma(2)},X_{n+1}]X_{\sigma(3)}\ldots X_{\sigma(n)} &+ \ldots +X_{\sigma(1)}\ldots
 X_{\sigma(n-1)}[X_{\sigma(n)},X_{n+1}]
\end{align*}
it follows easily that every multilinear polynomial $h$ satisfies
(cf.~[BCM, p.~170])
\begin{align}\label{e1}
\begin{split}
&[h(X_1,\ldots,X_n),X_{n+1}] = h([X_1,X_{n+1}],X_2,\ldots,X_n)\\
 & + h(X_1,[X_2,X_{n+1}],X_3,\ldots,X_n) + \ldots
+ h(X_1,\ldots,X_{n-1},[X_n,X_{n+1}]).
\end{split}
\end{align}

In order to state our first theorem we  have to recall a definition and record an elementary lemma which will be used frequently in the sequel.

\begin{definition}
An algebra $\A$ endowed with the {\it Lie product}
$$
[x,y]:=xy-yx \quad \text{for } x,y\in\A
$$
is a Lie algebra and the ideals of $\A$ with respect to this product
are called {\it Lie ideals} of $\A$.
\end{definition}

Thus, a Lie ideal of $\A$ is a linear subspace $\cL$ of $\A$ such that $[\cL,\A]\subseteq\cL$.

\begin{lemma} \label{R1}
Let $\V$ be a linear space over $\FF$, and let $\U$ be its subspace. Suppose that $c_0,c_1,\ldots,c_n\in \V$ are such that
\begin{equation}\label{eq:R1}
\sum_{i=0}^n \lambda^i c_i \in \U
\end{equation}
holds for at least $n+1$ different scalars $\lambda$. Then each $c_i\in \U$.
\end{lemma}

\begin{proof}
Let
$\lambda_{\ell}\in\FF$, $\ell=0,\ldots,n$, be different elements in $\FF$ satisfying \eqref{eq:R1}. 
Then 
\begin{equation}\label{eq:vandermonde}
\sum_{i=0}^n \lambda_\ell^i \overline c_i=0
\end{equation}
in the vector space $\V/\U$, where $v\mapsto\overline v$ denotes
the quotient mapping $\V\to \V/\U$.
The system \eqref{eq:vandermonde}$_{\ell=0,\ldots,n}$ can be equivalently written
in matrix form as
$$
\begin{bmatrix}
1 & \lambda_0 & \cdots & \lambda_0^n \\
\vdots& \vdots & \ddots & \vdots \\
1 & \lambda_n & \cdots & \lambda_n^n
\end{bmatrix}
\begin{bmatrix}
\overline c_0 \\ \vdots\\ \overline c_n
\end{bmatrix}
=\begin{bmatrix}
0 \\ \vdots\\ 0
\end{bmatrix}.
$$
The Vandermonde matrix on the left hand side 
is invertible as its determinant
is 
$$
\prod_{0\leq i< j\leq n} (\lambda_i-\lambda_j).
$$
Thus
$\overline c_i=0$, i.e., $c_i\in \U$ for all $i$.
\end{proof}

\begin{theorem} \label{T1}
Let $\FF$ be an infinite field, let $\A$ be an $\FF$-algebra, and let $\cL_1,\ldots,\cL_n$ be Lie ideals of $\A$. Then
for every $f = f(X_1,\ldots,X_n) \in \FF\langle \bar{X}\rangle$,  $\Span f(\cL_1,\ldots,\cL_n)$  is again a Lie ideal of $\A$.
\end{theorem}

\begin{proof}  We can write $f=f_0+f_1+\ldots +f_m$ where  $f_i$ is the sum of all monomials of $f$ that have degree $i$ in $X_1$.
Note that
$$f(\lambda a_1,a_2,\ldots,a_n) =  \sum_{i=0}^m \lambda^i f_i(a_1,\ldots,a_n) \in \Span f(\cL_1,\ldots,\cL_n)$$
 for all $\lambda\in \FF$ and all $a_i\in \cL_i$, and so $ f_i(a_1,\ldots,a_n) \in \Span f(\cL_1,\ldots,\cL_n)$ by Lemma \ref{R1}. Repeating the same argument with respect to other variables we see that values of  each of the multihomogeneous components of $f$ lie in   $\Span f(\cL_1,\ldots,\cL_n)$. But then there is no loss of generality in assuming that $f$ itself is multihomogeneous. Accordingly, we can write
$$
f= h(X_1,\ldots,X_1,X_2,\ldots,X_2,\ldots,X_n,\ldots,X_n)
$$
where $h\in \FF\langle \bar{X}\rangle$ is multilinear, $X_1$ appears $k_1$ times, $X_2$ appears $k_2$ times, etc.
Considering $f(a_1 + \lambda a_1',a_2,\ldots,a_n)$ we thus arrive at the relation $\sum_{i=0}^{k_1} \lambda^i c_i \in  \Span f(\cL_1,\ldots,\cL_n)$, where, in particular,
\begin{align*}
c_1 &= h(a_1',a_1,\ldots,a_1,a_2,\ldots,a_2,\ldots,a_n,\ldots,a_n)\\
 &+ h(a_1,a_1',a_1,\ldots,a_1, a_2,\ldots,a_2,\ldots,a_n,\ldots,a_n)\\
 &+\ldots + h(a_1,\ldots,a_1,a_1', a_2,\ldots,a_2,\ldots,a_n,\ldots,a_n).
\end{align*}
By Lemma \ref{R1}, each $c_i$, including of course $c_1$, belongs to $\Span f(\cL_1,\ldots,\cL_n)$; here, $a_1,a_1'\in \cL_1$ $a_2\in \cL_2,\ldots,a_n\in \cL_n$ are arbitrary elements. Similar statements can be established with respect to other variables.

 Now,
using \eqref{e1} we see that for all $a_i\in \cL_i$ and $b\in \A$ we have
\begin{align*}
[f(a_1,\ldots,a_n),b] &= h([a_1,b],a_1,\ldots,a_1,a_2,\ldots,a_2,\ldots,a_n,\ldots,a_n)\\
+ \ldots &+ h(a_1,\ldots,a_1,[a_1,b],a_2,\ldots,a_2,\ldots,a_n,\ldots,a_n)
\\
+ \ldots &+
 h(a_1,\ldots,a_1,[a_2,b],a_2,\ldots,a_2,\ldots,a_n,\ldots,a_n)\\
 + \ldots &+
 h(a_1,\ldots,a_1,a_2,\ldots,a_2,[a_2,b],\ldots,a_n,\ldots,a_n)\\
+ \ldots &+
 h(a_1,\ldots,a_1,a_2,\ldots,a_2,\ldots,[a_n,b],a_n\ldots,a_n)\\
+ \ldots &+
 h(a_1,\ldots,a_1,a_2,\ldots,a_2,\ldots,a_n\ldots,a_n,[a_n,b]).
\end{align*}
Let us point out that $[a_i,b]\in \cL_i$ since $\cL_i$ is a Lie ideal of $\A$.
In view of the above observation $c_1\in  \Span f(\cL_1,\ldots,\cL_n)$ it follows that the sum of the first $k_1$ summands that involve $[a_1,b]$ lies in  $\Span f(\cL_1,\ldots,\cL_n)$. Similarly we see that the sum of summands  involving $[a_2,b]$ lies in  $\Span f(\cL_1,\ldots,\cL_n)$, etc. Accordingly, $[f(a_1,\ldots,a_n),b]\in  \Span f(\cL_1,\ldots,\cL_n)$, proving that   $\Span f(\cL_1,\ldots,\cL_n)$ is a Lie ideal of $\A$.
\end{proof}

A very special case of Theorem \ref{T1}, where $f = [[X_1,X_2],X_2]$, was noticed in the recent paper \cite[Lemma 4.6]{BKS} as an auxiliary, but important result needed for describing Lie ideals of tensor products of algebras.

\subsection{The involution case}

For dealing with polynomial values in
algebras with involution we introduce the analogue
of a free algebra in the category of algebras with involution.
Let $\FF$ be a field with an involution $\ast$.
By $\FF\langle \bar{X},\bar X^*\rangle$ we denote the free $*$-algebra 
over $\FF$
generated by $\bar{X} = \{X_1,X_2,\ldots\}$, i.e., 
the $\FF$-algebra of all polynomials in noncommuting variables $X_i,X_j^*$. Further, by $\Sym$ we denote the set of all symmetric, and by
 $\Skew$ we denote  the
set of all skew-symmetric polynomials in $\FF\langle\bar{X},\bar X^*\rangle$ (with respect to the canonical involution, of course). By the {\em degree} of $X_i$ in a monomial $M\in \FF\langle\bar{X},\bar X^*\rangle$ we shall mean the number of appearances of  $X_i$ {\em or} $X_i^*$ in $M$. For example, both $X_1^2$ and $X_1X_1^*$ have degree $2$ in $X_1$.     
The concepts of (multi)homogeneity
and (multi)linearity of polynomials in $\FF\langle\bar{X},\bar X^*\rangle$ are defined accordingly. For example,
$X_1X_2X_1^* + X_2^*X_1^2$ is multihomogeneous and linear in $X_2$.

   Let  $\A$ be an algebra with involution $\ast$
    and let $f = f(X_1,\ldots,X_n,X_1^*,
\ldots, X_n^*)\in \FF\langle \bar{X},\bar X^*\rangle$.
If $\cL_1,\ldots,\cL_n$ are 
subsets of $\A$, then by $f(\cL_1,\ldots,\cL_n)$ we denote the set of all values $f(a_1,\ldots,a_n,a_1^*,\ldots,a_n^*)$ 
   with $a_i\in \cL_i$, $i=1,\ldots,n$. Again, if   $\cL_i=\A$ for every $i$, then we simply write $f(\A)$ instead of $f(\A,\ldots,\A)$.

Theorem \ref{T1} does not hold for polynomials in  $\FF\langle \bar{X},\bar X^*\rangle$. For example, if $f= X_1+X_1^*$, then (assuming $\Char (\FF)\ne 2$) $f(\A) = \cS$ and so $\Span f(\A)$ is only exceptionally a Lie ideal of $\A$. However, it does satisfy a weaker version of the definition of a Lie ideal: while it is, in general, not closed under commutation with elements from $\cS$, it is certainly closed under commutation with elements from $\K$ since 
$[\cS,\K]\subseteq \cS$. Subspaces satisfying this property will be one of the central topics of this paper.

\begin{definition}
A linear subspace $\cL$ of an algebra $\A$ with involution will be called a {\em Lie  skew-ideal} of $\A$ if 
$[\cL,\K]\subseteq\cL$.
\end{definition}

\begin{theorem} \label{prop:T1}
Let $\FF$ be an infinite field with $\Char (\FF)\ne 2$, let $\A$ be an $\FF$-algebra with
involution, and let $\cL_1,\ldots,\cL_n$ be Lie skew-ideals of $\A$. Then
for every  $f = f(X_1,\ldots,X_n,X_1^*,\ldots,X_n^*) \in 
\FF\langle \bar{X},\bar X^* \rangle$,  $\Span f(\cL_1,\ldots,\cL_n)$  is again a Lie skew-ideal of $\A$.
\end{theorem}

\begin{proof}
The proof is almost the same as the proof of Theorem \ref{T1}, so we only point out the necessary modifications.

The first part of the proof based on applications of Lemma \ref{R1} is literally the same, except that instead of scalars in $\FF$ one should deal with scalars from the subfield 
$\FF_0$ of all symmetric elements of $\FF$.
Since $[\FF:\FF_0]\le 2$, $\FF_0$ is also an infinite field, and so all arguments still work.

Let  $h=h(X_1,\ldots,X_n,X_1^*,\ldots,X_n^*) \in 
\FF\langle \bar{X},\bar X^* \rangle$ be a multilinear polynomial. The formula 
 \eqref{e1} does not hold for $h$ (not even if $h=X_1^*$). However, using $[X_i^*,X_{n+1}-X_{n+1}^*] = [X_i, X_{n+1}-X_{n+1}^*]^*$
one easily derives the following analogous formula
\begin{align*}
\begin{split}
&[h(X_1,\ldots,X_n,X_1^*,\ldots,X_n^*),X_{n+1}-X_{n+1}^*]\\
 =&  h([X_1,X_{n+1}-X_{n+1}^*],X_2,\ldots,X_n,[X_1,X_{n+1}-X_{n+1}^*]^*,X_2^*,\ldots,X_n^*)\\
  +& h(X_1,[X_2,X_{n+1}-X_{n+1}^*],X_3,\ldots,X_n,X_1^*,[X_2,X_{n+1}-X_{n+1}^*]^*,X_3^*,\ldots,X_n^*)\\ 
 +& \ldots+ h(X_1,\ldots,X_{n-1},[X_n,X_{n+1}-X_{n+1}^*],X_1^*,\ldots,X_{n-1}^*,[X_n,X_{n+1}-X_{n+1}^*]^*).
\end{split}
\end{align*}
Using this the proof is just a simple adaptation of the proof of  Theorem \ref{T1}. Here one also has to note that
every element in $\K$ is of the form $a-a^*$, $a\in\A$. 
\end{proof}

\section{Lie Ideals and Lie Skew-Ideals} \label{sec3}

The aim of this section is to describe Lie ideals and Lie skew-ideals in prime PI algebras. Let us recall that an algebra is said to be {\em prime} if the product of any of its two nonzero ideals is always nonzero. Prime PI algebras can be embedded (in a particularly nice way) into finite dimensional  central simple algebras (over a certain field extension of the base field), so in our arguments we shall mostly deal with the latter. In fact, dealing only with finite dimensional  simple algebras would not make arguments much more complicated, but would simplify somewhat the exposition. Anyhow, we have decided to consider prime PI algebras because of applications in Section \ref{sec4} and because of the fact that there exist important examples of such algebras that are not simple - for instance, the algebra of generic matrices considered in Section \ref{sec5}.

The concept of a Lie ideal is a classical one, and the result obtained in the first subsection below is not particularly surprising. The bulk of the section is devoted to Lie skew-ideals.

\subsection{Lie ideals in prime PI algebras}
The following result is folklore.

\begin{lemma} \label{Lfolk}
Let $\A=M_d(\FF)$, $d\ge 2$, and suppose that  $d\ne 2$ or $\Char(\FF)\ne 2$. Then $\A$ has exactly four Lie ideals: $0$, $\Z$, $[\A,\A]$  and $\A$.
\end{lemma}

Here, the center $\Z$ is equal to $\FF$, the set of all scalar matrices, and $[\A,\A]$  is the set of all commutators
$[A,B]$, $A,B \in \A$, or equivalently, the set of all matrices with zero trace.

A general remark about notation: if $\U$ and $\V$ are subspaces of an algebra $\A$, then by $[\U,\V]$ we denote the linear span of all commutators $[u,v]$, $u\in\U$, $v\in\V$. By chance in the case of $\A=M_d(\FF)$ the linear space $[\A,\A]$ coincides with the set of all commutators $[A,B]$ \cite{Sho,AM}, but in general this is not true.

One can prove Lemma \ref{Lfolk} by a direct computation. On the other hand, the lemma follows immediately from a substantially more general result by Herstein \cite[Theorem 1.5]{H2} stating that under very mild assumptions a Lie ideal of a simple algebra  $\A$ either contains $[\A,\A]$ or is contained in $\Z$.
We also remark that the case when $d=2$ and $\Char(\FF)=2$ is really exceptional, see \cite[p. 6]{H2}.

Now assume that  $\A$ is a prime PI algebra.  Then $\Z\ne 0$ and the {\em central closure} (i.e., a central localization, also called the algebra of central quotients \cite[\S 1.7]{Rowen}) $\wt{\A}$ of $\A$  consists of  elements of the form $z^{-1}a$ where $a\in \A$ and $0\ne z\in\Z$. Furthermore, $\wt{\A}$  is a finite dimensional central simple algebra over the field of fractions $\wt{\Z}$ of $\Z$. This is a version of Posner's theorem together with Rowen's sharpening, see for example
\cite[Theorem 1.7.9]{Rowen}. Of course, $\wt{\Z}$ is a field extension of $\FF$ and so they have the same characteristic. Given a subset $\V$ of $\A$, we shall write  $\wt{\V}$ for the linear span of $\V$ over $\wt{\Z}$.

\begin{proposition} \label{ppip}
Let $\A$ be a prime $\FF$-algebra such that $\dim_{\wt{\Z}}\wt{\A}\ne 4$ or  $\Char(\FF)\ne 2$. If $\cL$ is a Lie ideal of $\A$, then $\wt{\cL}$ is either $0$, $\wt{\Z}$, $[\wt{\A},\wt{\A}]$ or $\wt{\A}$.
\end{proposition}

\begin{proof}

Let $\ov{\Z}$ be the algebraic closure of $\wt{\Z}$. We now form the scalar extension $\ov{\A} = \wt{\A}\otimes_{\wt{\Z}}\ov{\Z}$ which is, as a finite dimensional central simple algebra over an algebraically closed field
$\ov{\Z}$, isomorphic to $M_d(\ov{\Z})$ where $d= \sqrt{\dim_{\wt{\Z}}\wt{\A}}$.  Thus  $d\ne 2$ 
if $\dim_{\wt{\Z}}\wt{\A}\ne 4$.

Observe  that $\wt{\cL}$ is a Lie ideal of $\wt{\A}$, and hence $\ov{\cL} = \wt{\cL}\otimes \ov{\Z}$ is a Lie ideal of $\ov{\A}$.  Lemma \ref{Lfolk} tells us that $\ov{\cL}$ is either $0$, $\ov{\Z}$,   $[\ov{\A},\ov{\A}]$ or $\ov{\A}$. Note that
 \begin{align} \label{elis}
 \begin{split}
  0 &= 0\otimes \ov{\Z}, \,\,\,\, \ov{\Z} =  \wt{\Z}\otimes \ov{\Z}, \,\,\,\, 
   [\ov{\A},\ov{\A}] = [\wt{\A},\wt{\A}] \otimes \ov{\Z}, \,\,\,\, \ov{\A} = \wt{\A}\otimes \ov{\Z}.
 \end{split}
 \end{align}
We now make a small digression and record the following easily proven fact: if $\W$ and $\V$ are $\wt{\Z}$-subspaces of $\wt{\A}$ and $\W\otimes \ov{\Z} = \V\otimes \ov{\Z}$, then $\W = \V$. Accordingly, since
$\ov{\cL} = \wt{\cL}\otimes \ov{\Z}$ is equal to one of the sets listed in \eqref{elis}, it follows that $\wt{\cL}$ is either $0$, $\wt{\Z}$,  $[\wt{\A},\wt{\A}]$ or $\wt{\A}$.
\end{proof}

\begin{remark}\label{Rcs}
If $\A$ itself is a finite dimensional central simple $\FF$-algebra, then this result gets a simpler form. Namely, in this case $\wt{\Z}=\Z=\FF$,  $\wt{\A} = \A$,  and  moreover $\wt{\V} = \V$  for every linear subspace $\V$ of $\A$. \end{remark} 

For more details about Lie ideals in simple algebras we refer the reader to \cite{H2}. A more recent reference is  the paper \cite{BKS} in which  Lie ideals are thoroughly studied in both algebraic and analytic setting.  

\subsection{General remarks on Lie skew-ideals} Let $\A$ be a  $\ast$-algebra over a field $\FF$ with $\Char(\FF)\ne 2$.  Every Lie ideal of $\A$ is also a Lie skew-ideal of $A$, while the converse is not true
in general. For example, $\cS$ and $\K$ are Lie skew-ideals, which are only rarely Lie ideals.
Obviously,  Lie skew-ideals are closed under sums and intersections. Further, if $\cL_1$ and
$\cL_2$ are  Lie skew-ideals, then $[\cL_1,\cL_2]$ is also a Lie skew-ideal. This can be easily checked by using the Jacobi identity.

 Let us mention eight examples of Lie
skew-ideals: $0$, $\Z$, $\K$, $[\cS,\K]$, $\cS$, $\Z$ + $\K$,  $[\A,\A]$, and $\A$.
As indicated above, there are other natural examples. The reasons for pointing out these eight examples will become clear in the sequel.

For subspaces of $\K$  the notion of a Lie skew-ideal
coincides with the standard and extensively 
studied notion of a Lie ideal of $\K$. For a simple algebra $\A$ with $\dim_{\Z} \A > 16$, a classical theorem by Herstein states that every Lie ideal of $\K$ either contains $[\K,\K]$ or is contained in $\Z$ \cite[Theorem 2.12]{H2}.  The following example justifies the dimension restriction.

\begin{example}\label{ex2}
If $\A=M_4(\FF)$, $\Char (\FF)\ne 2$, endowed with the transpose involution, then $\K$ can can be written as a Lie theoretic direct sum of two simple  Lie algebras, $\K=\K_1\oplus \K_2$. Each $\K_i$ is  3-dimensional;  a basis of $\K_1$ is $\{E_{12} - E_{21} + E_{34} - E_{43}, E_{13} - E_{31} + E_{42} - E_{24},  E_{14} - E_{41} + E_{23} - E_{32}\}$,  and  a basis of $\K_2$ is $\{E_{12} - E_{21} - E_{34} + E_{43}, E_{13} - E_{31} - E_{42} + E_{24},  E_{14} - E_{41} - E_{23} + E_{32}\}$. Thus, $\K_1$ and $\K_2$ are Lie ideals of $\K$ (and hence Lie skew-ideals of $\A$) which are neither contained in $\Z$ nor do they contain $[\K,\K]$ (which is equal to $\K$ in this example).
\end{example}

 Somewhat less known is Herstein's result which treats linear subspaces 
$\cL$ of $\cS$ satisfying $[\cL,[\K,\K]]\subseteq \cL$ \cite[Theorem 2.1]{H1}. Again assuming the simplicity of $\A$ and some additional mild technical conditions, this result says that $\cL$ either contains $[\cS,\K]$ or is contained in $\Z$. Of course this result also  covers Lie skew-ideals of $\A$ that are contained in $\cS$.  

Now let $\cL$ be a general Lie skew-ideal. If $\cL^* =\cL$ then $\cL = \cL\cap\cS \oplus \cL\cap\K$, and for $\cL\cap\cS$ and $\cL\cap \K$ we can use Herstein's aforementioned  results.  However, not every Lie skew-ideal is invariant under $\ast$.

\begin{example} \label{ex1}
Let $\cL$ be the one-dimensional subspace of $M_2(\FF)$
generated by $L=E_{11} + E_{12} - E_{21} + E_{22}$. Note that $\cL$ is a Lie skew-ideal of $M_2(\FF)$
with respect to the  transpose involution, but is not invariant under this involution.
\end{example}

In what follows we shall see that this example is a rather exceptional one. Nevertheless, it seems that Herstein's theorems are not directly applicable to our purposes. Not only because of the $\ast$-invariance problem, but also since we wish to obtain a precise \emph{description} of all Lie skew-ideals rather than just   information about certain inclusions. This seems to be out of reach in such a general class as is the class of simple algebras. But we shall confine ourselves to a more special class of prime PI  algebras - these algebras being close to finite dimensional simple algebras. Still, Herstein's theory has been useful for us philosophically. It indicates that Lie skew-ideals  are treatable.

Let us briefly discuss another question that naturally appears in connection with Lie skew-ideals, and which we find interesting in its own right. Assume that $\A$ has an identity element $1$ (in general we do not assume this in advance) and let $\U$ be the set of all unitary elements in $\A$, $\U = \{u\in\A\,|\,u^* = u^{-1}\}$. This question concerns
  the relation between
Lie skew-ideals  and  subspaces  of $\A$ that are closed under conjugation with unitary elements, i.e., subspaces $\cL$ of $\A$ such that $u\cL u^*\subseteq \cL$ for every $u\in\U$.
 This is an analogue to the problem of the relation between
Lie ideals  and  subspaces closed under conjugation with invertible elements (in other words, subspaces invariant under all inner automorphisms). One of the basic results in the latter area says that a closed linear subspace of a Banach algebra must be a Lie ideal if it closed under conjugation with invertible elements. The same proof shows the following.

\begin{proposition}\label{Lbanach}
Let $\A$ be a real or complex Banach algebra with $\mathbb R$-linear involution $*$. If  a closed linear subspace $\cL$ of $\A$ is closed under conjugation with unitaries, then $\cL$ is a  Lie skew-ideal of $\A$.
\end{proposition}

\begin{proof}
If $k\in \K$, then $e^{tk}\in \U$ for every $t\in \mathbb R$. Therefore, for every $t\ne 0$ and $x\in\cL$, $\cL$ contains the element
$$
 \frac{1}{t}\Bigl(e^{tk}x(e^{tx})^* - x\Bigr) = \frac{1}{t}\Bigl(e^{tk}xe^{-tx} - x\Bigr) = [k,x] + \frac{t}{2!}[k,[k,x]] + \frac{t^2}{3!}[k,[k,[k,x]]] + \ldots
$$
(The second equality is a special case of the Baker-Campbell-Hausdorff
formula.)
Since $\cL$ is closed it follows that $$
[k,x]  = \lim_{t\rightarrow 0}  \frac{1}{t}\Bigl(e^{tk}x(e^{tx})^* - x\Bigr)\in\cL.
$$
\end{proof}

Algebraic versions of this propositions cannot be obtained so easily. Namely, in a purely algebraic setting the set 
of unitary elements
 can be very small, and so  
$u\cL u^*\subseteq \cL$ for all $u\in \U$ may trivially hold.

\begin{example}\label{exfreealg}
\mbox{}
\begin{enumerate}[\rm (1)]
\item
If $\A=\FF\langle \bar{X},\bar X^*\rangle$ is a free $*$-algebra, then $\U\subseteq \FF$, and so every subspace of $\A$ is 
closed under conjugation with unitaries. But of course not every subspace is a Lie skew-ideal.
\item
For a finite dimensional example of characteristic $3$,
let $\FF_3=\{0,1,2\}$ denote the field on $3$ elements and consider
$M_2(\FF_3)$ endowed with the transpose involution. Then
$\K$ is spanned by
$k=\begin{bmatrix} 0 & 2\\
1& 0\end{bmatrix}$
and
$$
\U=\left\{
\begin{bmatrix}
 \lambda & 0 \\
  0 & \mu
  \end{bmatrix},
\begin{bmatrix}
 0 & \lambda \\
  \mu & 0
  \end{bmatrix}
\mid
\lambda,\mu\in\FF_3\setminus\{0\}
\right\}.
  $$
  Now it is easy to construct
  examples of subspaces closed under unitary conjugation
  that are not Lie skew-ideals. For instance, take
  $\FF_3 \begin{bmatrix} 0&1\\
  1&0 \end{bmatrix}$.
\item
We conclude by presenting an example of a slightly different flavor.
Cohn \cite[Exercise 2.1.10]{Coh} has
constructed a division algebra (necessarily of characteristic
$2$) with only one unitary element.
Like in (1) this gives rise to an abundance of examples of subspaces
closed under unitary conjugation that are not Lie skew-ideals.
\end{enumerate}
\end{example}

There is, however, the following nice result by Lanski (which we state using our terminology): If $\A$ is an algebraic $*$-algebra over an infinite field $\FF$ with $\Char(\FF)\ne 2$, then every linear subspace of $\A$ which is closed under conjugation with unitary elements is  a  Lie skew-ideal of $\A$ \cite[Theorem 1]{Lan}. The converse is not true. Indeed, one can check that Lie skew-ideals from Examples \ref{ex2} and \ref{ex1} are not closed under conjugation with unitaries. On the other hand, Lie skew-ideals that will be important for us, namely $0$, $\Z$, $\K$, $[\cS,\K]$, $\cS$, $\Z$ + $\K$,  $[\A,\A]$ and $\A$,   are all closed  under conjugation with unitaries.

Let us finally mention that for every $f  \in \FF\langle \bar{X}\rangle$ and every algebra $\A$,  $\Span f(\A)$ is closed under conjugation with invertible elements; moreover, it is invariant under every algebra endomorphism of $\A$. Similarly,  for every $f  \in \FF\langle \bar{X}, \bar{X}^* \rangle$,  $\Span f(\A)$ is closed under conjugation with unitary elements, and moreover, it is invariant under every algebra $\ast$-endomorphism of $\A$. In view of these observations we have been in fact hesitating at the early stage of this work whether the definition of a Lie skew-ideal should also involve the  conjugation with unitaries.  However, it has turned out that this would lead to certain technical difficulties, and so we have decided to focus on commutation with skew-symmetric elements only.

\subsection{Lie skew-ideals in matrix algebras}
The purpose of this section is to describe Lie skew-ideals in matrix algebras with respect to two basic involutions, the transpose and the usual symplectic involution. Let us at the beginning present these notions in a more general framework.

\begin{definition}
Let $\A$ be a central simple $\ast$-algebra of degree $d$, i.e., of dimension $d^2$
over its center $\Z$. Then $*$ is called \textit{orthogonal} if 
\[\dim_\Z\cS=\frac{d(d+1)}{2}\]
and \textit{symplectic} if
\[\dim_\Z\cS=\frac{d(d-1)}{2}.\]
\end{definition}

Symplectic involutions only exist for even $d$. 
For a full account on (central simple) algebras with involutions we refer
the reader to \cite{boi}.

The basic example of an orthogonal involution on the algebra $\A=M_d(\FF)$ is  the {\em transpose involution}, $A\mapsto A^t$. The {\em usual symplectic involution} on $\A=M_d(\FF)$ is defined  when $d$ is even, $d = 2d_0$, as follows:
$$
\begin{bmatrix} A&B \\C &D \end{bmatrix}^* = \begin{bmatrix} D^t&-B^t \\-C^t &A^t \end{bmatrix}\quad\mbox{where $A,B,C,D\in M_{d_0}(\mathbb F)$.}
$$ 

\begin{definition}
An involution  on an algebra $\A$ is said to be
\textit{of the first kind} if it fixes $\Z$ pointwise and 
\textit{of the second kind} otherwise. Involutions of the second kind are also called \textit{unitary involutions}.
\end{definition}

Both the transpose and the usual symplectic involution are of course involutions of the first kind. 

\begin{lemma}\label{lmatriket}
Let $\A=M_d(\FF)$ be endowed with the transpose involution, and  
let $\Char(\FF) \ne 2,3$. If  $d\ne 2,4$, then $0$, $\Z$, $\K$, $[\cS,\K]$, $\cS$, $\Z$ + $\K$,  $[\A,\A]$, and $\A$ are the only Lie skew-ideals of $\A$.
\end{lemma}

\begin{proof} Let us begin by noting that $\Z$ consists of all scalar matrices, $[\cS,\K]$ consists of all symmetric matrices with trace $0$, and $[\A,\A]$ consists of all matrices with trace $0$.

Since $d\ne 2,4$, $\K$ is a simple Lie algebra. This is  well-known and easy to see (see for example \cite[p.~443]{BMMb}). Given a Lie skew-ideal  $\cL$ of $\A$, we have that
 $\cL\cap\K$ is a Lie ideal of $\K$, and hence either   $\cL\cap\K = 0$ or $\cL\cap\K =\K$. That is,
\begin{equation}\label{eLK}
\cL\cap\K = 0 \quad\mbox{or}\quad \K\subseteq\cL.
\end{equation}

Let us  first consider the case where $\cL\subseteq \Z+\K$. If $\cL\subseteq \Z$, then of course either 
$\cL =0$ or $\cL =\Z$. 
If $\cL\not\subseteq \Z$, then $\cL$ contains a matrix $\lambda I + K_0$ where $\lambda\in \FF$ and $0\ne K_0\in \K$. Picking $K_1\in \K$ which does not commute with $K_0$ it follows that $0\ne [K_0,K_1] = [\lambda I + K_0,K_1]\in \cL\cap\K$. Therefore $\K\subseteq\cL$ by \eqref{eLK}. But then either
$\cL = \K$ or $\cL=\Z+\K$.

Assume from now on that  $\cL\not\subseteq \Z+\K$. Therefore there exists $A=(a_{ij})\in\cL$ such that for some 
$i\ne j$, either $\alpha = a_{jj}-a_{ii}\ne 0$ or $\beta = a_{ij}+a_{ji}\ne 0$. Since for every $K\in \K$ also 
$K^3\in\K$, we have
$$K^2AK - KAK^2 = \frac{1}{3}\Bigl( [[[A,K],K],K] - [A,K^3]\Bigr)\in \cL.$$
For  $K= E_{ij}-E_{ji}$ we get 
\begin{equation}\label{eaij}
\alpha(E_{ij} + E_{ji}) + \beta (E_{ii} - E_{jj})\in\cL.
\end{equation}
Pick $k$ different from $i$ and $j$ (recall that $d\ne 2$!). Since $E_{jk}-E_{kj}\in\cL$, it follows that $\cL$ contains
$$
[[\alpha(E_{ij} + E_{ji}) + \beta (E_{ii} - E_{jj}), E_{jk}-E_{kj}],E_{jk}-E_{kj}] = -\alpha(E_{ij} + E_{ji}) + 2\beta(E_{jj}-E_{kk}).
$$
Using this together with \eqref{eaij} it follows that $\beta(E_{ii} + E_{jj} - 2E_{kk})\in \cL$, and hence also
$$
\beta(E_{ik} + E_{ki}) = \frac{1}{3}[ \beta(E_{ii} + E_{jj} - 2E_{kk}), E_{ik} - E_{ki}]\in\cL.
$$
If $\beta\ne 0$, then this yields. $E_{ik} + E_{ki}\in \cL$. If, however, $\beta =0$, then $\alpha\ne 0$ and hence $E_{ij} + E_{ji}\in \cL$ by \eqref{eaij}. Thus, in any case $\cL$ contains a matrix of the form $E_{uv} + E_{vu}$ with $u\ne v$. We claim that this implies that $\cL$ contains all matrices of the form $E_{pq} + E_{qp}$ with $p\ne q$. Indeed, if $\{p,q\}\cap \{u,v\}=\emptyset$, then this follows from  $E_{pq} + E_{qp} = [[E_{uv} + E_{vu}, E_{vp} - E_{pv}], E_{uq} - E_{qu}]$, and if $\{p,q\}\cap \{u,v\}\ne\emptyset$, then the proof is even easier. Consequently,
$E_{qq} - E_{pp} = \frac{1}{2}[E_{pq} + E_{qp}, E_{pq} - E_{qp}]\in\cL$. Note that all these relations can be summarized as
\begin{equation} \label{eS0}
[\cS,\K]\subseteq \cL.
\end{equation}

Suppose that $\cL\cap\K = 0$. We claim that in this case $\cL\subseteq\cS$. Indeed, if this was not true, then $\cL$ would contain a matrix $K_0+ S_0$ with $0\ne K_0\in\K$ and $S_0\in \cS$. Picking $K_1\in \K$ that does not commute with $K_0$ it then follows from \eqref{eS0} that $0\ne [K_0,K_1] = [K_0+S_0,K_1] - [S_0,K_1] \in \cL\cap \K$, a contradiction. Thus $[\cS,\K]\subseteq \cL\subseteq\cS$ and so either $\cL=[\cS,\K]$ or $\cL=\cS$. 

It remains to consider the case where $\cL\cap\K \ne 0$. In this case $ \K\subseteq \cL$ by \eqref{eLK}. Since
$\cL$ also contains $[\cS,\K]$ and since $[\cS,\K] + \K = [\A,\A]$, it follows that $[\A,\A]\subseteq \cL\subseteq\A$. But then  either $\cL=[\A,\A]$ or $\cL=\A$. 
\end{proof}

The cases where $d=2$ or $d=4$ are indeed exceptional; see  Examples \ref{ex1} and \ref{ex2}.

Our next aim is to prove a version of Lemma \ref{lmatriket} for the usual symplectic involution. For this we need the following lemma which describes the structure of certain subspaces of $M_d(\FF)$ that are in particular Lie skew-ideals of $M_d(\FF)$ with respect to the transpose involution. Since the restrictions $d\ne 2,4$ and   
$\Char(\FF) \ne 3$ are unnecessary in this situation,  we cannot apply  Lemma \ref{lmatriket}. In any case a direct computational proof could be easily given. However, a result by Montgomery \cite[Corollary 1]{Mont} describing additive subgroups $\M$ of simple rings $\A$ with involution satisfying $a\M a^*\subseteq \M$ for all $a\in \A$ will make it  possible for us to use a shortcut. This result implies that if $\A$ is a simple algebra over a field $\FF$ with 
$\Char(\FF) \ne 2$, the involution $\ast$ is of the first kind, and $\M$ is such a linear subspace of $\A$, then 
$\M$ must be either  $0$,  $\K$, $\cS$, or $\A$.

\begin{lemma}\label{lmatrikeU}
Let $\A=M_d(\FF)$ be endowed  with the transpose involution, and  
let $\Char(\FF) \ne 2$.
If $\M$ is a linear subspace of $\A$ such that $MA^t + AM\in \M$ for all $M\in\M$ and $A\in\A$, then $\M$ is either
$0$,  $\K$, $\cS$, or $\A$.
\end{lemma}

\begin{proof}
From the identity   
$$
AMA^t = \frac{1}{2}\Bigl( \bigl((MA^t + AM)A^t + A (MA^t + AM)\bigr) - \bigl(M(A^2)^t + A^2M\bigr) \Bigr)
$$
it follows that $AMA^t \in \M$ for all $A\in \A$ and $M\in \M$. Therefore the result follows immediately from \cite[Corollary 1]{Mont}.
\end{proof}

\begin{lemma}\label{lmatrikes}
Let $\A=M_{2d_0}(\FF)$, let $\ast$ be the usual symplectic involution on $\A$, and  
let $\Char(\FF) \ne 2$. Then $0$, $\Z$, $\K$, $[\cS,\K]$, $\cS$, $\Z$ + $\K$,  $[\A,\A]$, and $\A$ are the only Lie skew-ideals of $\A$.
\end{lemma}

\begin{proof}
Set $\A_0 = M_{d_0}(\FF)$ and let  $\K_0$ and $\cS_0$ denote the sets of symmetric and skew-symmetric matrices in $\A_0$ with respect to the transpose involution.  Note that
$\K$  consists of all matrices of the form 
$$
\begin{bmatrix} A&S \\T &-A^t \end{bmatrix} \quad\mbox{where $A\in \A_0$, $S, T\in\cS_0$,}
$$
and $\cS$  consists of all matrices of the form 
$$
\begin{bmatrix} A&K \\L &A^t \end{bmatrix} \quad\mbox{where $A\in \A_0$, $K, L\in\K_0$.}
$$

Let  $\cL$ be a Lie skew-ideal of $\A$, and let
 $\begin{bmatrix} A&B \\C &D \end{bmatrix}\in \cL$. Commuting this matrix with  $\begin{bmatrix} I&0 \\0 &-I \end{bmatrix}\in\K$ it follows that  $\begin{bmatrix} 0&-B \\C &0 \end{bmatrix}\in \cL$. Furthermore, commuting the latter matrix with $\begin{bmatrix} I&0 \\0 &-I \end{bmatrix}$ one easily shows that actually both
$\begin{bmatrix} 0&B \\0 &0 \end{bmatrix}$ and $\begin{bmatrix} 0&0 \\C &0 \end{bmatrix}$ belong to $\cL$.  Thus, we have
\begin{equation} \label{ABCD}
\begin{bmatrix} A&B \\C &D \end{bmatrix}\in \cL\Rightarrow \begin{bmatrix} A&0 \\0 &D \end{bmatrix},
\begin{bmatrix} 0&B \\0 &0 \end{bmatrix}, \begin{bmatrix} 0&0 \\C &0 \end{bmatrix} \in\cL.
\end{equation}

 Let $\M_0$ be the set of all $M\in \A_0$ such that  $\begin{bmatrix} 0&M \\0 &0 \end{bmatrix}\in\cL$. Commuting this matrix with $\begin{bmatrix} A&0 \\0 &-A^t \end{bmatrix}\in \K$ it follows that $\M_0$, considered as a subspace of $\A_0$, satisfies the condition of Lemma \ref{lmatrikeU}. Therefore $\M_0$ is $0$,  $\K_0$, $\cS_0$, or $\A_0$. Each of these four cases shall be considered separately.
 
 Assume that $\M_0 = 0$. From \eqref{ABCD} we see that then  any matrix in $\cL$ is of the form $\begin{bmatrix} A&0 \\C &D \end{bmatrix}$. Commuting such a matrix with $\begin{bmatrix} 0&S \\0 &0 \end{bmatrix}\in \K$ it follows that $AS= SD$ for all $S\in \cS_0$. It is easy to see that this is possible only if $A = D$ is a scalar matrix. Consequently, commuting $\begin{bmatrix} 0&0 \\C &0 \end{bmatrix}$ with $\begin{bmatrix} 0&I \\0 &0 \end{bmatrix}$ it follows that $C=-C$, i.e., $C=0$. Therefore $\cL$ consists only of scalar matrices. There are just two possibilities: either $\cL =0$ or $\cL = \Z$. 
 
 Next we consider the case where $\M_0 = \K_0$. Pick $K\in\K_0$ and $S\in\cS_0$. Commuting $\begin{bmatrix} 0&K \\0 &0 \end{bmatrix}\in \cL$ with $\begin{bmatrix} 0&0 \\S &0 \end{bmatrix}\in \K$ it follows that $\begin{bmatrix} KS&0 \\0&-SK  \end{bmatrix}\in\cL$. It is easy to see that every matrix in $\A_0$ of the form $KS$ has trace $0$, and conversely, every matrix in $\A_0$ with trace $0$ is a linear span of matrices of the form $KS$. Therefore $\cL$ contains all matrices $\begin{bmatrix} A&0 \\0&A^t  \end{bmatrix}$ with $A\in [\A_0,A_0]$. Now take any matrix in $\cL$ of the form $\begin{bmatrix} A&0 \\0 &D \end{bmatrix}$. Its commutator with  $\begin{bmatrix} 0&0 \\S &0 \end{bmatrix}\in \K$ is $\begin{bmatrix} 0&AS-SD \\0 &0 \end{bmatrix}$. Since this matrix must be in $\cL$ it follows that $AS-SD\in\K_0$ for every $S\in\cS_0$. This condition can be rewritten as $S(A^t - D) + (A^t - D)^tS =0$  for every $S\in\cS_0$. It is easy to see that this forces $A^t = D$.  Therefore the ``diagonal part" of $\cL$ consists only of matrices of the form $\begin{bmatrix} A&0 \\0 &A^t \end{bmatrix}$, and there are two possibilities: either all such matrices with an arbitrary $A\in \A_0$ are in $\cL$, or only all such matrices with the restriction that $A$ has trace $0$, i.e., $A\in [\A_0,\A_0]$. It remains to examine the ``lower corner" part. Pick $\begin{bmatrix} 0&0 \\C &0 \end{bmatrix}\in \cL$. Commuting it with $\begin{bmatrix} 0&I \\0 &0 \end{bmatrix}\in \K$ we get $\begin{bmatrix} -C&0 \\0 &C \end{bmatrix}\in \cL$. But then $C$ must lie in $\K_0$. Conversely, as the commutator of $\begin{bmatrix} A&0 \\0 &A^t \end{bmatrix}\in \cL$ with $\begin{bmatrix} 0&0 \\I &0 \end{bmatrix}\in \K$ is $\begin{bmatrix} 0&0 \\A^t-A &0 \end{bmatrix}$, and since every $K \in\K_0$ can be written as $K=A^t - A$ with $A\in [\A_0,\A_0]$, it follows that $\cL$ contains all matrices 
 $\begin{bmatrix} 0&0 \\K &0 \end{bmatrix}$ with $K\in\K_0$. 
 We can now gather all the information derived in the following
conclusion:   $\cL$ either consists of all matrices  $\begin{bmatrix} A&K \\L &A^t \end{bmatrix}$  with $A\in \A_0$, $S, T\in\K_0$ or of all such matrices with $A\in [\A_0,A_0]$, $S, T\in\K_0$. In the first case 
 $\cL= \cS$ and in the second case $\cL = [\cS,\K]$.
 
 The cases where $\M_0 = \cS_0$ or $\M_0 = \A_0$ can be treated  similarly as the $\M_0 = \K_0$ case. One can show that   $\M_0 = \cS_0$ implies that  $\cL= \K$ or $\cL = \Z + \K$, and $\M_0 = \A_0$ implies that  $\cL= [\A,\A]$ or $\cL =\A$. There are some differences compared to the case just treated, but the necessary modifications are quite obvious. Therefore we omit the details.
 \end{proof} 

\subsection{Lie skew-ideals in prime PI algebras}
The above results  make it possible for us to describe Lie skew-ideals in prime PI algebras with involution. The description depends on the kind of an involution.

As in the first subsection on Lie ideals, we denote  by $\wt{\A}$ the
central closure of  a prime PI algebra $\A$, and by  $\wt{\Z}$ the field of
fractions of $\Z$. By $\wt{\V}$ we denote the linear span of $\V\subseteq\A$
over $\wt{\Z}$.

\begin{theorem} \label{TPI}
Let  $\A$ be a prime PI algebra with involution of the first kind, and let $\cL$ be a  Lie skew-ideal of $\A$.
If  $\dim_{\wt{\Z}} {\wt{\A}} \ne 4, 16$ and  $\Char(\FF) \ne 2,3$, then $\wt{\cL}$ is either $0$, $\wt{\Z}$, $\wt{\K}$, $[\wt{\cS},\wt{\K}]$, $\wt{\cS}$, $\wt{\Z}$ + $\wt{\K}$,  $[\wt{\A},\wt{\A}]$ or $\wt{\A}$.
\end{theorem}

\begin{proof}
As in the proof of Proposition \ref{ppip} we denote the algebraic closure of $\wt{\Z}$ by
$\ov{\Z}$, and  form the scalar extension $\ov{\A} = \wt{\A}\otimes_{\wt{\Z}}\ov{\Z}$ which is  isomorphic to $M_d(\ov{\Z})$ where $d= \sqrt{\dim_{\wt{\Z}}\wt{\A}}$. We can extend $\ast$ to  an involution (also of the first kind) of
$\wt{\A}$ according to $(z^{-1}a)^* = z^{-1}a^*$, and then further to an involution of $\ov{\A}$ (of the first kind)  by
$(z^{-1}a\otimes \lambda)^\ast = z^{-1}a^*\otimes \lambda$.  Note that
$\wt{\cS}$ is the set of symmetric elements of $\wt{\A}$, and
$\wt{\K}$ is the set of skew-symmetric  elements of $\wt{\A}$. Further,
 the set of symmetric elements $\ov{\cS}$ of $\ov{\A}$ is equal to
$\wt{\cS}\otimes \ov{\Z}$,
 and the set of skew-symmetric elements $\ov{\K}$ of $\ov{\A}$ is equal to
$\wt{\K}\otimes \ov{\Z}$.

Observe first that $\wt{\cL}$ is a Lie skew-ideal of $\wt{\A}$, and hence $\ov{\cL} = \wt{\cL}\otimes \ov{\Z}$ is a Lie skew-ideal of $\ov{\A}$. We now apply the description of an involution on $\ov{\A}$: there exists a set of matrix units $\{e_{ij}\}$ in  $\ov{\A}$ such that $\ast$ is either the transpose or the usual symplectic involution relative to $\{e_{ij}\}$ \cite[Corollary 4.6.13]{BMMb}. The condition that $\dim_{\wt{\Z}} {\wt{\A}} \ne 4, 16$ implies that $\ov{\A}$ is not isomorphic to $M_2(\ov{\Z})$ or $M_4(\ov{\Z})$.  We may now use Lemmas \ref{lmatriket} and \ref{lmatrikes}, and conclude that
 $\ov{\cL}$ is either $0$, $\ov{\Z}$, $\ov{\K}$, $[\ov{\cS},\ov{\K}]$, $\ov{\cS}$, $\ov{\Z}$ + $\ov{\K}$,  $[\ov{\A},\ov{\A}]$ or $\ov{\A}$. Note that
 \begin{align*} 
 \begin{split}
  0 &= 0\otimes \ov{\Z}, \,\,\,\, \ov{\Z} =  \wt{\Z}\otimes \ov{\Z}, \,\,\,\, \ov{\K} =  \wt{\K}\otimes \ov{\Z}, \,\,\,\,
  [\ov{\cS},\ov{\K}] =  [\wt{\cS},\wt{\K}] \otimes \ov{\Z}, \\
   \ov{\cS} = & \wt{\cS}\otimes \ov{\Z}, \,\,\,\, \ov{\Z} + \ov{\K} = (\wt{\Z} + \wt{\K} )\otimes \ov{\Z}, \,\,\,\,
   [\ov{\A},\ov{\A}] = [\wt{\A},\wt{\A}] \otimes \ov{\Z}, \,\,\,\, \ov{\A} = \wt{\A}\otimes \ov{\Z}.
 \end{split}
 \end{align*}
Hence it follows, just as in the proof of Proposition \ref{ppip}, that 
$\wt{\cL}$ is either $0$, $\wt{\Z}$, $\wt{\K}$, $[\wt{\cS},\wt{\K}]$, $\wt{\cS}$, $\wt{\Z}$ + $\wt{\K}$,  $[\wt{\A},\wt{\A}]$ or $\wt{\A}$.
\end{proof}

\begin{remark}\label{rKKK}
Assume the conditions of Theorem \ref{TPI}. Consider $\cL = [\K,\K]$. Clearly, $\cL$ is a Lie skew-ideal of $\A$. Since
 $\wt{\cL} = [\wt{\K},\wt{\K}]$ is contained in $\wt{\K}$, we see from Theorem \ref{TPI} that we have just two possibilities: either $[\wt{\K},\wt{\K}] =0$ or $[\wt{\K},\wt{\K}] = \wt{\K}$. As one can easily check by passing to 
 $\ov{\A}$, the first possibility is possible only when $\A$ is commutative (or when $\dim_{\wt{\Z}} {\wt{\A}} = 4$, but this case was excluded by the assumption of Theorem \ref{TPI}). Therefore  $[\wt{\K},\wt{\K}] = \wt{\K}$. This  (probably known) fact will be needed in the proof of Theorem \ref{thm:T2invo}.
\end{remark}

Let us now consider the simpler case when $\ast$ is of the second kind, i.e., $\ast$ is not  the identity on $\Z$.  With reference to the above notation we have the following result.

\begin{theorem} \label{TPIsecond}
Let  $\A$ be a prime PI algebra with involution of the second kind, and let $\cL$ be a  Lie skew-ideal of $\A$.
If   $\Char(\FF) \ne 2$, then $\wt{\cL}$ is either $0$, $\wt{\Z}$,  $[\wt{\A},\wt{\A}]$ or $\wt{\A}$.
\end{theorem}

\begin{proof}
The involution on $\A$ extends to $\wt{\A}$ in the obvious way, $(z^{-1}a)^* = {z^*}^{-1}a^*$.
Since $\ast$ is of the second kind, there exists $z\in \Z$ such that $w=z-z^*\ne 0$. Thus $w$ is nonzero skew-symmetric element in $\Z$. Pick $x\in \cL$ and $a\in \A$. We can write $a = s+k$ where $s\in\cS$ and $k\in \K$; indeed, we take $s=\frac{a+a^*}{2}$, $k=\frac{a-a^*}{2}$.  Clearly, $ws\in \K$ and so $[x,ws] \in \cL$, and of course also $[x,k]\in\cL$. But then $[x,a] = w^{-1}[x,ws] + [x,k]\in \wt{\cL}$. This proves that $[\cL,\A]\subseteq \wt{\cL}$, which readily implies that $[\wt{\cL},\wt{\A}]\subseteq \wt{\cL}$. That is, $\wt{\cL}$ is a Lie ideal of $\wt{\A}$. Now apply Proposition \ref{ppip}.
\end{proof}

\begin{corollary}\label{cor:csaLieSkew}
Let  $\A$ be a central simple algebra with involution $*$,
and let $\cL$ be a  Lie skew-ideal of $\A$.
\begin{enumerate}[\rm (1)]
\item
Suppose $*$ is  of the first kind.
If  $\dim_{{\Z}} {{\A}} \ne 4, 16$ and  $\Char(\FF) \ne 2,3$, then ${\cL}$ is either $0$, ${\Z}$, ${\K}$, $[{\cS},{\K}]$, ${\cS}$, ${\Z}$ + ${\K}$,  $[{\A},{\A}]$ or ${\A}$.
\item
Suppose $*$ is  of the second kind.
If   $\Char(\FF) \ne 2$, then $\cL$ is either $0$, ${\Z}$,  $[{\A},{\A}]$ or 
${\A}$.
\end{enumerate}
\end{corollary}

\section{Classifying Polynomials and Tracial Nullstellens\"atze} \label{sec4}

The purpose of this section is to classify the polynomials in $\FF\langle \bar{X}\rangle$ and in $\FF\langle \bar{X}, \bar{X}^*\rangle$ with respect to their values on prime PI algebras (with involution), and then as corollaries of these classification results derive what we call ``tracial Nullstellens\"atze''. 

\subsection{Cyclic equivalence} The following notion was introduced in \cite{KS}.

\begin{definition} \label{dce}
We say that polynomials 
$f,g$ in $\FF\langle \bar{X}\rangle$ (resp.~in $\FF\langle \bar{X},\bar{X}^*\rangle$) are {\em cyclically equivalent} (notation $f \csim
 g$) if $f -g$ is a sum of commutators in $\FF\langle \bar{X}\rangle$ (resp.~in $\FF\langle \bar{X},\bar{X}^*\rangle$).
\end{definition}

The next remark shows that cyclic equivalence can be checked easily
and that it is ``stable'' under scalar extensions in
the following sense: Given a field extension $\FF\subseteq \mathbb K$
and $f,g\in \FF\ax$, then $f\csim g$ in $\FF\ax$ if and only if 
$f\csim g$ in
$\mathbb K\ax$.
We note it holds verbatim for $\FF\langle\bar X,\bar X^*\rangle$ but is 
stated here only for $\FF\ax$.

\begin{remark}
\mbox{}
\begin{enumerate}[\rm (a)]
\item Two words $v,w\in\ax$ are cyclically equivalent
if and only if there are words
$v_1,v_2\in\ax$ such that $v=v_1v_2$ and $w=v_2v_1$.
\item Two polynomials $f=\sum_{w\in\ax}a_ww$ and $g=\sum_{w\in\ax}b_ww$
($a_w,b_w\in\FF$) are cyclically equivalent if and only if for each 
$v\in\ax$,
$$\sum_{w\csim v}a_w=\sum_{w\csim v}b_w.$$
\end{enumerate}
\end{remark}

The next two lemmas are simple, but 
essential for this paper.

\begin{lemma}\label{L1}
Let $f = f(X_1,\ldots,X_n)\in \FF\langle \bar{X}\rangle$. If $f$ is linear in $X_n$, then there exists 
$g=g(X_1,\ldots,X_{n-1})\in \FF\langle \bar{X}\rangle$ such that $f\csim gX_n$.
\end{lemma}

\begin{proof}
It suffices to treat the case when $f$ is a monomial, that is $f= mX_nm'$ where $m$ and $m'$ are monomials in $X_1,\ldots,X_{n-1}$. But then the result follows immediately from the identity
$mX_nm'  - m'mX_n = [mX_n,m']$.
\end{proof}

\begin{lemma}\label{L1*}
Let $f = f(X_1,\ldots,X_n,X_1^*,\ldots, X_n^*)\in \FF\langle \bar{X},\bar X^*\rangle$.
 If $f$ is linear in $X_n$, then there exist 
$g=g(X_1,\ldots,X_{n-1},X_1^*,\ldots, X_{n-1}^*)\in \FF\langle \bar{X},\bar X^*\rangle$ and $ g' = g'(X_1,\ldots,X_{n-1},X_1^*,\ldots, X_{n-1}^*) \in \FF\langle \bar{X},\bar X^*\rangle$ such that $f\csim gX_n + X_n^*g'$.
\end{lemma}

\begin{proof}
 The proof is basically the same as the proof of Lemma \ref{L1}. It suffices to consider the case where $f$ is a monomial. If  $f= mX_nm'$ then use $mX_nm'  - m'mX_n = [mX_n,m']$, and if 
$f= mX_n^*m'$ then use $mX_n^*m'  - X_n^*m'm = [m,X_n^*m']$.
 \end{proof}

\subsection{Polynomials in $\FF\langle \bar{X}\rangle$}
Consider first $\A = M_d(\FF)$. Let $f\in \FF\langle \bar{X}\rangle$.
Theorem \ref{T1} and Lemma \ref{Lfolk} imply that $\Span f(\A)$
 can be either $0$, $\Z$, $[\A,\A]$  or
$\A$. Each of the four possibilities indeed occurs. Finding polynomials $f$ such that $\Span f(\A)$ is either  $[\A,\A]$   or
$\A$ is trivial (say, take  $X_1X_2-X_2X_1$ and $X_1$). Since $\A$ is a PI algebra, we can find (nonzero) polynomials 
$f$ such that $\Span f(\A) = 0$. The existence of polynomials
$f$ such that  $\Span f(\A)=\Z$ is nontrivial. These are  the so-called {\em central polynomials}, i.e., polynomials which are not identities on $\A$ but all their values lie in $\Z$. In the early 70's Formanek \cite{F} and Razmyslov \cite{R} independently proved that for every $d\ge 2$ there exist central polynomials on $M_d(\FF)$.

Instead of $M_d(\FF)$ we could consider any prime PI algebra $\A$, just that we then have to deal with the linear span of  $\Span f(\A)$ over the field of fractions $\wt{\Z}$ of $\A$. Again we arrive at four possibilities.
Our goal is to determine when each of them occurs.

We use the same notation as above, i.e., the central closure of $\A$ is denoted by $\wt{\A}$, and 
the linear span of $\V\subseteq\A$ over $\wt{\Z}$ is denoted by
$\wt{\V}$. 

\begin{theorem} \label{T2}
Let  $\A$ be a noncommutative prime PI algebra, let  $f \in \FF\langle\bar{X}\rangle$, and let us write $\cL:=\Span f(\A)$. If $\Char (\FF)=0$, then exactly one of the following four possibilities holds:
\begin{enumerate}[\rm (i)]
\item $f$ is an identity of $\A$; in this case $\wt{\cL}=0$;
\item  $f$ is a central polynomial of $\A$; in this case $\wt{\cL}=\wt{\Z}$;
\item  $f$ is not an identity of $\A$, but is  cyclically equivalent to an identity of $\A$; in this case $\wt{\cL}=[\wt{\A},\wt{\A}]$;
\item  $f$ is not a central polynomial of $\A$ and is not cyclically equivalent to an identity of $\A$; in this case $\wt{\cL}=\wt{\A}$.
\end{enumerate}
\end{theorem}

\begin{proof} Theorem \ref{T1} and
Proposition \ref{ppip} tell us that $\wt{\cL}$ is either $0$, $\wt{\Z}$, $[\wt{\A},\wt{\A}]$ or $\wt{\A}$. 

 We claim that $\wt{\Z}\cap [\wt{\A},\wt{\A}] = 0$. A standard scalar extension argument shows that it suffices to prove this for the case where  $\wt{\A} = M_d(\FF)$. In this case the desired conclusion follows from the fact that the trace of the identity matrix is $d $ which is, as   $\Char(\FF) =0$, different from $0$.

Suppose first that $f$ is cyclically equivalent to an identity. Then 
$f(\A)\subseteq [\A,\A]$ and hence $\wt{\cL}\subseteq [\wt{\A},\wt{\A}]$. Since $\wt{\Z}\cap [\wt{\A},\wt{\A}] = 0$, there are  only two possibilities: either $\wt{\cL}=0$ or $\wt{\cL}=[\wt{\A},\wt{\A}]$.
 If $f$ itself is an identity, then of course (i) holds. If $f$ is not an identity, then $\cL\ne 0$ and so (iii) must hold.

 Assume now that $f$ is not cyclically equivalent to an identity. If $f$ is a central polynomial, then (ii) holds. Assume therefore that $f$ is not a central polynomial. We must show that $\wt{\cL}=\wt{\A}$. Obviously, $\wt{\cL}\ne 0$ and $\wt{\cL}\ne \wt{\Z}$.
   We still have to eliminate the possibility that $\wt{\cL}=[\wt{\A},\wt{\A}]$.
 Assume that this possibility actually occurs, so in particular $f(\A) \subseteq [\wt{\A},\wt{\A}]$.
 Writing $f$ as a sum of multihomogeneous polynomials, and then arguing as at the beginning of the proof of Theorem \ref{T1} we see that each of these homogeneous components has the same property that its values lie in $[\wt{\A},\wt{\A}]$.  It is obvious that at least one of these summands is not cyclically equivalent to an identity. Thus, there exists a multihomogeneous polynomial, let us call it $h=h(X_1,\ldots,X_n)$, which
 is not cyclically equivalent to an identity and has the property $h(\A)  \subseteq [\wt{\A},\wt{\A}]$.   We will show that this is impossible by induction on the degree  of $h$ with respect to $X_n$.  Let us denote this degree by $k$. If $k=1$, then we can use Lemma \ref{L1} to find a polynomial $g=g(X_1,\ldots,X_{n-1})$ such that
  $h \csim gX_n$. Consequently, $(gX_n)(\A)\subseteq [\wt{\A},\wt{\A}]$. Pick $a_1,\ldots,a_{n-1}\in\A$
  and write $w= g(a_1,\ldots,a_{n-1})$. Then $wx\in [\wt{\A},\wt{\A}]$ for every $x\in\A$, which clearly implies that 
  the same is true for every $x\in\wt{\A}$. 
  If $w\ne 0$, then because of the simplicity of $\wt{\A}$ there exist $u_i,v_i\in \wt{\A}$ such that $1= \sum_i u_iwv_i$. But then 
$$1 = \sum_i [u_i,wv_i] + w\sum_i v_iu_i \in [\wt{\A},\wt{\A}],$$
contradicting $\wt{\Z}\cap [\wt{\A},\wt{\A}] = 0$. Thus $w=0$, i.e.,  $g(a_1,\ldots,a_{n-1})=0$ for all $a_i\in\A$.
 That is,  $g$, and hence also $gX_n$, is an identity of $\A$. This contradicts our assumption that $h$ is not cyclically equivalent to an identity. Now let $k >1$ and consider the polynomial
 \begin{align*}
  h'(X_1,\ldots,X_{n},X_{n+1}) &= h(X_1,\ldots,X_{n-1},X_n + X_{n+1})\\& - h(X_1,\ldots,X_{n-1},X_n) - h(X_1,\ldots,X_{n-1}, X_{n+1}).
 \end{align*}
 Obviously the values of $h'$ also lie in
 $[\wt{\A},\wt{\A}]$, and so the same is true for each of multihomogeneous components of $h'$.  Since the degree 
 in $X_n$  of each of these components is smaller than $k$, the induction assumption implies that 
 each of them is  cyclically equivalent to an identity. But then $h'$ itself  is cyclically equivalent to an identity. However, since
 $$
 h(X_1,\ldots,X_n) = \frac{1}{2^k -2} h(X_1,\ldots,X_{n},X_{n})
 $$
 it follows that $h$ is also cyclically equivalent to an identity - a contradiction.
\end{proof}

We record the following two easily obtained corollaries related to \cite[Theorem 2.1]{KS}.
We call them tracial Nullstellens\"atze; the first one deals with
the non-dimensionfree setting and the second one is dimensionfree.

\begin{corollary} \label{C1}
Let $d\ge 2$,  let  $\Char (\FF)=0$  and let $f = f(X_1,\ldots,X_n) \in \FF\langle\bar{X}\rangle$.
Then $\tr(f(A_1,\ldots,A_n)) =0$ for all $A_i\in M_d(\FF)$ if and only if $f$ is cyclically equivalent to an identity of $ M_d(\FF)$.
\end{corollary}

\begin{corollary} \label{C2}
Suppose  $\Char (\FF)=0$  and let $f = f(X_1,\ldots,X_n) \in \FF\langle\bar{X}\rangle$.
Then $\tr(f(A_1,\ldots,A_n)) =0$ for all $A_i\in M_d(\FF)$ and all $d\ge 2$ if and only if $f\csim 0$.
\end{corollary}

\subsection{Polynomials in $\FF\langle \bar{X}, \bar X^*\rangle$} \label{sec:starEval}
Our aim now is  to obtain versions of Theorem \ref{T2} for polynomials in $\FF\langle \bar{X}, \bar X^*\rangle$.
The situation is  easier
for involutions of the second kind.

We continue to use the notation from the previous subsection.

\begin{theorem}  \label{cor:T2star}
Let  $\A$ be a noncommutative prime PI algebra with involution of the second kind, let  $f \in \FF\langle\bar{X}, \bar{X}^*\rangle$, and let us write $\cL:=\Span f(\A)$. If $\Char (\FF)=0$, then exactly one of the following four possibilities holds:
\begin{enumerate}[\rm (i)]
\item $f$ is an identity of $\A$; in this case $\wt{\cL}=0$;
\item  $f$ is a central polynomial of $\A$; in this case $\wt{\cL}=\wt{\Z}$;
\item  $f$ is not an identity of $\A$, but is  cyclically equivalent to an identity of $\A$; in this case $\wt{\cL}=[\wt{\A},\wt{\A}]$;
\item  $f$ is not a central polynomial of $\A$ and is not cyclically equivalent to an identity of $\A$; in this case $\wt{\cL}=\wt{\A}$.
\end{enumerate}
\end{theorem}

\begin{proof}
Not only the formulation, also the proof of this theorem is almost literally the same as the proof of Theorem \ref{T2}. Let us therefore just point out a few instances where small changes are necessary. Firstly, one of course has to use Theorem \ref{TPIsecond} (rather than Proposition \ref{ppip}) to conclude that $\wt{\cL}$ is $0$, $\wt{\Z}$, $[\wt{\A},\wt{\A}]$ or $\wt{\A}$. Secondly, for the reduction to  multihomogeneous polynomials one has to make use only of those scalars that the involution on $\FF$ fixes. They also form a field with characteristic $0$, so the same argument works.
Thirdly and finally, instead of Lemma \ref{L1} one has to use Lemma \ref{L1*} and thereby conclude that 
 $h\csim gX_n + X_n^*g'$, and hence $(gX_n + X_n^*g')(\A)\subseteq [\wt{\A},\wt{\A}]$. Since the involution is of the second kind this clearly implies that both $(gX_n)(\A)$ and $(X_n^*g')(\A)$ lie in $[\wt{\A},\wt{\A}]$. From this point on the necessary changes are completely obvious.
\end{proof}

For an involution of the first kind the situation is somewhat more
complicated since  Theorem \ref{TPI} yields eight possible classes.

For the ease of exposition we introduce some notation to be used in the next theorem.
Let $\A$ be a  PI algebra endowed with a (fixed) involution $\ast$.  By $\id(\A)$ we denote the set
of all polynomial identities of $\A$ in $\FF\langle\bar{X},\bar X^*\rangle$. At this point it seems appropriate to mention that if an algebra satisfies a nontrivial identity in $\FF\langle\bar{X},\bar X^*\rangle$, then it also satisfies a nontrivial 
identity in $\FF\langle\bar{X}\rangle$ \cite{Amit}; this is why in the $\ast$-algebra context we confine ourselves to  (usual) PI algebras. Next, by $\cen(\A)$ we denote
the set of all central polynomials of $\A$ in $\FF\langle\bar{X}\rangle$.
Note that  $\id(\A)$ and $\cen(\A)$  depend on the involution chosen. 

\begin{theorem} \label{thm:T2invo}
Let  $\A$ be a prime PI algebra with involution of the first kind, let $f \in \FF\langle\bar{X},\bar X^*\rangle$, and let us write $\cL:=\Span f(\A)$.
If  $\dim_{\wt{\Z}} {\wt{\A}} \ne 1, 4, 16$ and  $\Char(\FF) =0$, then exactly one of the following eight possibilities holds:
\begin{enumerate}[\rm (i)]
\item $f\in\id(\A)$; in this case $\wt{\cL}=0$;
\item  $f\in\cen(\A)$; in this case $\wt{\cL}=\wt{\Z}$;
\item $f\in\Skew + \id(\A)$ and $f\not\in\id(\A)$; in this case $\wt{\cL}=\wt{\K}$;
\item $f\in\Skew + \cen(\A)$ and $f\not\in\cen(\A)$; in this case $\wt{\cL}=\wt{\Z} + \wt{\K}$;
\item $f\in\Sym + \id(\A)$, $f\not\in\id(\A)$ and $f$ is cyclically equivalent to an element of $\id(\A)$; in this case $\wt{\cL}=[\wt{\cS},\wt{\K}]$;
\item $f\in\Sym + \id(\A)$, $f\not\in\cen(\A)$ and $f$ is not cyclically equivalent to an element of $\id(\A)$; in this case $\wt{\cL}=\wt{\cS}$;
\item  $f\not\in\Sym + \id(\A)$, $f\not\in\Skew + \id(\A)$, 
 and $f+f^*$ is
cyclically equivalent to an element of $\id(\A)$; in this case $\wt{\cL}=[\wt{\A},\wt{\A}]$;
\item $f\not\in\Sym + \id(\A)$, $f\not\in\Skew + \id(\A)$, 
$f\not\in\Skew + \cen(\A)$ and $f+f^*$ is not cyclically equivalent to an element of $\id(\A)$; in this case $\wt{\cL}=\wt{\A}$.
\end{enumerate}
\end{theorem}

\begin{proof}
We start by remarking that
$\cL$ is a Lie skew-ideal of $\A$ by Theorem \ref{prop:T1}.
Therefore
$\wt{\cL}$ is either $0$, $\wt{\Z}$, $\wt{\K}$, $[\wt{\cS},\wt{\K}]$, $\wt{\cS}$, $\wt{\Z}$ + $\wt{\K}$,  $[\wt{\A},\wt{\A}]$ or $\wt{\A}$ by Theorem \ref{TPI}.

We divide the proof into two parts, (a) and (b), depending on whether or not $f+f^*$ is cyclically equivalent to an element of $\id(\A)$.

 (a) Assume that $f+f^*$ is cyclically equivalent to an identity. Then $f= \frac{f + f^*}{2} + \frac{f - f^*}{2}$ is a sum of an identity, commutators, and a skew-symmetric polynomial, and hence $f(\A)\subseteq [\A,\A]  + \K\subseteq 
 [\wt{\A},\wt{\A}] + \wt{\K}$. In Remark \ref{rKKK} we have showed that $ \wt{\K} = [ \wt{\K}, \wt{\K}]$. This forces  $f(\A)\subseteq  [\wt{\A},\wt{\A}]$, and consequently  $\wt{\cL}\subseteq [\wt{\A},\wt{\A}]$.
 
 Recall from the proof of Theorem \ref{T2} that $\wt{\Z}\cap [\wt{\A},\wt{\A}] = 0$.  
Therefore  $\wt{\cL}$ is neither $\wt{\Z}$, $\wt{\Z} + \wt{\K}$, $\wt{\cS}$ nor
 $\wt{\A}$. Thus $\wt{\cL}\in\{0,\wt{\K},[\wt{\cS},\wt{\K}],[\wt{\A},\wt{\A}]\}$. 
 If $f$ itself is an identity, then of course (i) holds. 
 Now suppose $f$ is not an identity. If $f\in\Skew + \id(\A)$, then (iii) holds.
 If $f\in\Sym + \id(\A)$, then (v) holds. Otherwise (vii) holds. Let us also point out that $f$ cannot belong to $\Skew + \cen(\A)$ if (vii) occurs.

(b) Now assume that $f+f^*$ is not cyclically equivalent to an identity. Let us first show that 
$\wt{\cL}\not\subseteq [\wt{\A},\wt{\A}]$. Suppose this is not true, that is, suppose $f(\A)\subseteq [\wt{\A},\wt{\A}]$. As a skew-symmetric polynomial, $f-f^*$  automatically satisfies  $(f-f^*)(\A)\subseteq\wt{\K}\subseteq [\wt{\A},\wt{\A}]$ by Remark \ref{rKKK}. But then $s=f+f^* = 2f - (f-f^*)$ has the same property, i.e., $s(\A)\subseteq [\wt{\A},\wt{\A}]$.  Suppose that $s$ is linear in $X_n$. Then Lemma \ref{L1*} tells us that there exist
$g=g(X_1,\ldots,X_{n-1},X_1^*,\ldots, X_{n-1}^*)\in \FF\langle \bar{X},\bar X^*\rangle$ and $g' = g'(X_1,\ldots,X_{n-1},X_1^*,\ldots, X_{n-1}^*) \in \FF\langle \bar{X},\bar X^*\rangle$ such that $s\csim gX_n + X_n^*g'$.
It is clear that then $(gX_n + X_n^*g')(\A)\subseteq [\wt{\A},\wt{\A}]$. Pick $a_1,\ldots,a_{n-1}\in\A$ and 
set $b=g(a_1,\ldots,a_{n-1},a_1^*,\ldots,a_{n-1}^*)$, $c=g'(a_1,\ldots,a_{n-1},a_1^*,\ldots,a_{n-1}^*)$. Then
$bx + x^*c\in [\wt{\A},\wt{\A}]$ for all $x\in\A$, and hence also for all $x\in\wt{\A}$. Consequently,
$$
(b+c^*)x = (bx + x^*c) + (c^*x - x^*c)\in [\wt{\A},\wt{\A}] + \wt{\K} =  [\wt{\A},\wt{\A}].
$$
Thus $w\wt{\A} \subseteq [\wt{\A},\wt{\A}]$ where $w = b +c^*$. As in the proof of Theorem \ref{T2} we see  that this yields $w=0$, i.e., 
$$g(a_1,\ldots,a_{n-1},a_1^*,\ldots,a_{n-1}^*) +
g'(a_1,\ldots,a_{n-1},a_1^*,\ldots,a_{n-1}^*)^*=0.$$
Since the $a_i$'s are arbitrary elements in $\A$, this means that
 $g + g'^*\in\id(\A)$. Thus 
 $$
 s\csim gX_n + X_n^*g' =  (-h^*X_n +X_n^* g') + (g+g'^*)X_n\in \Skew  +\id(\A).
 $$
Since $s = f+f^*\in \Sym$ and since both $\Skew$ and $\id(\A)$ are invariant under $\ast$, we now  arrive at the contradiction that $s$ is cyclically equivalent to an element in $\id(\A)$. Recall that this was derived under the assumption that $s$ is linear in $X_n$. The general case can be reduced to this one in the same way as in the proof of Theorem \ref{T2}. Therefore we have indeed $\wt{\cL}\not\subseteq [\wt{\A},\wt{\A}]$.

We now know that $\wt{\cL}\in\{\wt{\Z}, \wt{\cS}, \wt{\Z} + \wt{\K}, \wt{\A}\}$.
If $f\in \cen(\A)$, then (ii) holds. Suppose now that $f$ is not a central polynomial. If $f\in\Skew + \cen(\A)$, then (iv) holds. If  $f\in\Sym + \id(\A)$, then (vi) must hold. Otherwise we have (viii).

Due to the construction of the cases (i) - (viii) it is clear that
they are exhaustive and mutually exclusive.
\end{proof}

\begin{remark}
Let us mention again that  in finite dimensional central simple algebras our theorems get simpler forms. Roughly speaking, for these algebras the presence of $\,\,\wt{}\,\,$ is simply unnecessary in Theorems \ref{T2}, \ref{cor:T2star} and \ref{thm:T2invo}. That is, $\wt{\A} =\A$, $\wt{\Z} =\Z$, $\wt{\cS} =\cS$, etc. 
\end{remark}

We are now in a position to give the tracial Nullstellens\"atze
for free $*$-algebras:

\begin{corollary} \label{cor:C3}
Let $d\ne 1,2,4$,  let  $\Char (\FF)=0$  and let $f 
\in \FF\langle\bar{X},\bar X^*\rangle$ be a polynomial in $n$ variables.
Fix an involution $\ast$ on $M_d(\FF)$. If it is of the first kind, assume 
that $f\in\Sym$.
Then $\tr(f(A_1,\ldots,A_n,A_1^*,\ldots, A_n^*)) =0$ for all 
$A_i\in M_d(\FF)$ 
if and only if $f$ is cyclically equivalent to an identity of $ M_d(\FF)$.
\end{corollary}

\begin{corollary}[cf.~Theorem $2.1$ in \cite{KS}] \label{cor:C4}
Let  $\Char (\FF)=0$  and let $f
\in \FF\langle\bar{X},\bar X^*\rangle$ be a polynomial in $n$ variables.
Fix an involution $\ast$ on $M_d(\FF)$. If it is of the first kind, assume 
that $f\in\Sym$.
Then $\tr(f(A_1,\ldots,A_n,A_1^*,\ldots, A_n^*)) =0$ for all 
$A_i\in M_d(\FF)$ 
and all $d\ge 2$ if and only if $f\csim 0$.
\end{corollary}

\begin{remark}
The results given in this subsection 
can be easily extended to free algebras with involution 
generated by symmetric (or skew-symmetric) variables. 
\end{remark}

\section{Algebra of generic matrices} \label{sec5}

In this section we interpret some of our main results
 in
the algebra of generic matrices \cite{Rowen}. 
As is often the case, verifying a condition on 
values of a polynomial on $d\times d$ matrices
is conveniently done in the algebra of generic matrices.
Here we discuss ``having zero trace'' and 
present appropriate versions of 
Corollaries \ref{C1} (in Theorem \ref{thm:gm1}) and \ref{cor:C3}
(in Theorem \ref{thm:gm2}).

Somewhat related is the result of Amitsur and Rowen \cite{AR},
where they show that in central simple algebras an element
is a sum of (two) commutators if and only if its reduced trace
is zero; see also \cite[Appendix to 3]{AR} and
\cite{RR}.

\subsection{Algebra of generic matrices (without involution)}

Let $\zeta:=(\zeta_{ij}^{(\ell)}\mid1\leq i, j\leq d,\,\ell\in\N)$ 
denote commuting
variables and form the polynomial algebra $\FF[\zeta]$. 
Then the \textit{algebra of generic $d\times d$ matrices} 
$\GM_d(\FF)$ is the subalgebra of $M_d(\FF[\zeta])$ generated by the 
$d\times d$ matrices $Y_{\ell}:=\begin{bmatrix}
\zeta_{ij}^{(\ell)}\end{bmatrix}_{1\leq i,j\leq d}$, where 
$\ell\in\N$. Each $Y_{\ell}$ is
called a \textit{generic matrix}. Furthermore, 
$\GM_d(\FF,n)$
is used to denote the 
subalgebra generated by the $n$ generic matrices $Y_1,\ldots,Y_n$. 
The algebra of generic $d\times d$ matrices is a PI algebra and a domain. 
Moreover, $\GM_d(\FF)$ is isomorphic to $\FF\ax /\id_d$,
where $\id_d$ is the ideal of all polynomial identities of
$d\times d$ matrices.

$\GM_d(\FF)$ enjoys the following property: any algebra
homomorphism 
$${\rm ev}_{ a}:\FF[\zeta]\to\FF, \quad p\mapsto p( a)
$$ 
``lifts'' to a 
homomorphism of algebras
$\GM_d(\FF)\to M_d(\FF)$ by entrywise evaluation.
The image of an element $f\in\GM_d(\FF)$
under this map will be denoted simply by $f( a)$.

\begin{theorem}[Tracial Nullstellensatz for generic matrices]\label{thm:gm1}
Suppose $\Char(\FF)=0$ and let $f\in\GM_d(\FF)$. Then the
following are equivalent:
\begin{enumerate}[\rm (i)]
\item $f$ is a sum of commutators in $\GM_d(\FF)$;
\item $\tr(f)=0$;
\item $\tr ( f( a))=0$ for all $ a\in M_d(\FF)^\N$;
\item $f( a)$ is a sum of commutators in $M_d(\FF)$ for
all $ a\in M_d(\FF)^\N$.
\end{enumerate}
\end{theorem}

\begin{proof}
The equivalences (ii) $\Leftrightarrow$ (iii) and 
(iii) $\Leftrightarrow$ (iv) are obvious as is
the implication (i) $\Rightarrow$ (iii). For the proof
of (iii) $\Rightarrow$ (i) let $F=F(X_1,\ldots,X_n)\in\FF\ax$ denote
a preimage of $f$ under the homomorphism
$\FF\ax\to\GM_d(\FF)$. Then $f( a)=
F(Y_1( a),\ldots, Y_n( a))$. As
$ a$ runs through all of $M_d(\FF)^\N$,
$(Y_1( a),\ldots, Y_n( a))$ sweeps through
all $n$-tuples of $d\times d$ matrices over $\FF$.
By assumption, $\tr (F(A_1,\ldots, A_n))=0$
for all $A_i\in M_d(\FF)$. 
Hence Corollary \ref{C1} implies that $F$ is cyclically equivalent
to an identity of $M_d(\FF)$. Thus $f$ is a sum of commutators.
\end{proof}

It is clear that a similar statement holds for 
$f\in\GM_d(\FF,n)$.

\subsection{Algebra of generic matrices with involution}
Like in the classical construction of the algebra of generic
matrices,
it is possible to construct the algebra of generic matrices {\it with
involution} \cite{ps}. 
To each type of involution (orthogonal, symplectic
and unitary) an algebra of generic matrices
with involution can be associated. We proceed to describe details.
From now we assume that 
$\FF$ is a field of characteristic
$0$ with an involution.

Let $\FF\langle\bar X, \bar X^*\rangle$ be the free 
$*$-algebra
over $\FF$. 
For an involution of type ${\rm J}\in\{$symplectic, orthogonal, unitary$\}$,
let $\id_{{\rm J},d}\subseteq\FF\langle\bar X, \bar X^*\rangle$ denote the ideal of all identities satisfied by degree $d$ central simple
algebras with involution of type $\rm J$ \cite[\S 2]{Rowen}
(a word of caution: 
the notation has changed from Section \ref{sec:starEval} in order to emphasize
the dependence on the (type of) involution).
That is, $f=f(X_1,\ldots,X_k,X_1^*,\ldots
,X_k^*)\in\FF\langle\bar X, \bar X^*\rangle$ 
is an element of $\id_{{\rm J},d}$ if and only if for every
central simple algebra $\A$ with involution of type $\rm J$ of
degree $d$ and every $a_1,\ldots,a_k\in\A$, 
$$f(a_1,\ldots,a_k,a_1^*,\ldots,a_k^*)=0.$$
Then $\GM_d(\FF,{\rm J}):=
\FF\langle\bar X, \bar X^*\rangle/
\id_{{\rm J},d}$ is the {\it algebra of generic $d\times d$ 
matrices
with involution of type $\rm J$}.

Let $\zeta,Y_\ell$ be as above. The involution on $\FF$
is extended to $\FF[\zeta]$ by fixing the $\zeta_{ij}^{(\ell)}$
pointwise. 
\begin{enumerate}[\rm (1)]
\item If ${\rm J}=$ orthogonal or ${\rm J}=$ unitary, 
then $\GM_d(\FF,{\rm J})$ is canonically
isomorphic to the (unital) $\FF$-subalgebra of $M_d(\FF[\zeta])$ 
generated by the $Y_\ell$ and their transposes.
\item If ${\rm J}=$ symplectic, then $d$ is even, say $d=2d_0$, 
and the (unital) $\FF$-subalgebra of $M_{2d_0}(\FF[\zeta])$ 
generated by the $Y_\ell$ and their
images under the usual symplectic involution 
is (canonically) isomorphic to $\GM_{2d_0}(\FF,{\rm J})$.
\end{enumerate}

Let $\KK$ denote the set of all symmetric elements
of $\FF$.
Then every $*$-algebra homomorphism $\FF[\zeta]\to\FF$
is described by a point $a\in M_d(\KK)^\N$ and given by the images of 
$\zeta_{ij}^{(\ell)}$.
Hence it induces a $*$-algebra (evaluation) homomorphism 
$\GM_d(\FF,{\rm J})\to M_d(\FF)$ 
denoted by $g\mapsto g(a,a^*)$.
If $G\in \FF\langle \bar X,\bar X^*\rangle$ is a polynomial whose 
coset in $\GM_d(\FF,{\rm J})$
is represented by $g$, then $g(a,a^*)$ equals $G(a,a^*)$, the evaluation
of $G$ at the tuple of $d\times d$ matrices $a$.
This means that, as before, any $*$-algebra
homomorphism 
$${\rm ev}_{ a}:\FF[\zeta]\to\FF, \quad p\mapsto p( a, a^*)
$$ 
lifts to a 
$*$-homomorphism
$\GM_d(\FF,{\rm J})\to M_d(\FF)$ by entrywise evaluation.

\begin{theorem}[Tracial Nullstellensatz for generic matrices
with involution]\label{thm:gm2}
\hfill Suppose $\Char(\FF)=0$, fix a type ${\rm J}$ and let
$f\in\GM_d(\FF,{\rm J})$. 
Write $\KK$ for the set of symmetric elements of $\FF$. 
If ${\rm J}\neq$ unitary, assume moreover that
$f=f^*$.
Then the
following are equivalent:
\begin{enumerate}[\rm (i)]
\item $f$ is a sum of commutators in $\GM_d(\FF,{\rm J})$;
\item $\tr( f)=0$;
\item $\tr (f( a, a^*))=0$ for all $ a\in M_d(\KK)^\N$;
\item $f( a, a^*)$ is a sum of commutators in $M_d(\FF)$ for
all $ a\in M_d(\KK)^\N$.
\end{enumerate}
\end{theorem}

The proof is essentially the same as that of Theorem \ref{thm:gm1}
and is therefore omitted.

\end{document}